\documentclass[12pt,reqno]{amsart}
\usepackage{etoolbox}
\makeatletter
\patchcmd\maketitle
  {\uppercasenonmath\shorttitle}
  {}
  {}{}
\patchcmd\maketitle
  {\@nx\MakeUppercase{\the\toks@}}
  {\the\toks@}
  {}
  {}{}
\patchcmd\@settitle{\uppercasenonmath\@title}{\Large}{}{}
\patchcmd\@setauthors
  {\MakeUppercase{\authors}}
  {\authors}
  {}{}
\makeatother
\usepackage{amsmath,amssymb,amsthm}
\usepackage{color}
\usepackage{url}
\usepackage{tikz-cd}
\usepackage[utf8]{inputenc}
\usepackage[T1]{fontenc}
\newtheorem{theorem}{Theorem}[section]

\newtheorem{corollary}{Corollary}[section]

\newtheorem{lemma}{Lemma}[section]
\newtheorem{remark}{Remark}[section]

\numberwithin{equation}{section}

  \renewcommand{\thethqt}{\Alph{thqt}}
\usepackage[colorlinks=true]{hyperref}
\hypersetup{urlcolor=blue, citecolor=red , linkcolor= blue}
\usepackage[capitalise,noabbrev,nameinlink]{cleveref}      
\begin{document}
\address{$^{[1_a]}$ University of Monastir, Faculty of Economic Sciences and Management of Mahdia, Mahdia, Tunisia}
\address{$^{[1_b]}$ Laboratory Physics-Mathematics and Applications (LR/13/ES-22), University of Sfax, Faculty of Sciences of Sfax, Sfax, Tunisia}
\email{\url{kais.feki@fsegma.u-monastir.tn}\,;\,\url{kais.feki@hotmail.com}}



\date{\today}
\author[Kais Feki] {\Large{Kais Feki}$^{1_{a,b}}$}
\title[Further improvements of generalized numerical radius inequalities for Hilbert space operators]{Further improvements of generalized numerical radius inequalities for Hilbert space operators}

\maketitle

\begin{abstract}
Several new improvements of the $A$-numerical radius inequalities for operators acting on a semi-Hilbert space, i.e., a space generated
by a positive operator $A$, are proved. In particular, among other inequalities, we show that
\begin{align*}
\frac{1}{4}\|T^{\sharp_A} T+TT^{\sharp_A}\|_A
\leq\frac{1}{4}\Big(2\omega_A^2(T)+\gamma(T)\Big)
\leq \omega_A^2(T),
\end{align*}
where
$$\gamma(T)=\sqrt{\left(\|\Re_A(T)\|_A^2-\|\Im_A(T)\|_A^2\right)^2+4\|\Re_A(T)\Im_A(T)\|_A^2}.$$
Here $\omega_A(X)$ and $\|X\|_A$ denote respectively the $A$-numerical radius and the $A$-seminorm of an operator $X$. Also, $\Re_A(T):=\frac{T+T^{\sharp_A}}{2}$ and $\Im_A(T):=\frac{T-T^{\sharp_A}}{2i}$, where $T^{\sharp_A}$ is a distinguished $A$-adjoint operator of $T$. Further, some new refinements of the triangle inequality related to $\|\cdot\|_A$ are established.
\end{abstract}

\section{Introduction and Preliminaries}\label{s1}
Let $\mathcal{H}$ be a complex Hilbert space with inner product $\langle\cdot,\cdot\rangle$ and associated norm $\|\cdot\|$. The $C^*$-algebra of all bounded linear operators acting $\mathcal{H}$ will be denoted by $\mathbb{B}(\mathcal{H}).$ Let $T\in\mathbb{B}(\mathcal{H})$, the numerical radius and the usual operator norm of $T$ are defined respectively by
$$\omega(T)=\sup_{\|x\|=1}\left|\langle Tx,x\rangle\right|\;{\text{ and }}\;\|T\|=\sup_{\|x\|=1}\|Tx\|.$$
An operator $T\in\mathbb{B}(\mathcal{H})$ is said to be positive (denoted by $T\geq0$) if $\langle Tx,x\rangle \geq 0$ for all $x\in\mathcal{H}$. In all that follows, the range of every operator $T\in\mathbb{B}(\mathcal{H})$ is denoted by $\mathcal{R}(T)$, its null space by $\mathcal{N}(T)$ and $T^*$ is the adjoint of $T$. If $\mathcal{M}$ is an arbitrary linear subspace of $\mathcal{H}$, then $\overline{\mathcal{M}}$ denotes its closure in the norm topology of $\mathcal{H}$. Given a closed subspace $\mathcal{M}$ of $\mathcal{H}$, $P_{\mathcal{M}}$ stands for the orthogonal projection onto $\mathcal{M}$. For the rest of this paper, by an operator we mean a bounded linear operator and we assume that $A\in\mathbb{B}(\mathcal{H})$ is a nonzero positive operator. It is clear that $A$ induces a semi-inner product on $\mathcal{H}$ given by ${\langle x, y\rangle}_A = \langle Ax, y\rangle$ for all $x, y \in\mathcal{H}$. By ${\|\cdot\|}_A$ we denote the seminorm induced by ${\langle \cdot, \cdot\rangle}_A$, i.e. ${\|x\|}_A=\sqrt{{\langle x, x\rangle}_A}$ for every $x\in\mathcal{H}$. One can verify that ${\|x\|}_A = 0$ if and only if $x\in\mathcal{N}(A)$. This implies that ${\|\cdot\|}_A$ is a norm on $\mathcal{H}$ if and only if $A$ is injective. Further, we observe that the semi-Hibert space $(\mathcal{H}, {\|\cdot\|}_A)$ is a complete space if and only if $\mathcal{R}(A)$ is closed in $\mathcal{H}$. For $T\in \mathbb{B}(\mathcal{H})$, an operator $S\in \mathbb{B}(\mathcal{H})$
is called an $A$-adjoint operator of $T$ if the identity ${\langle Tx, y\rangle}_A = {\langle x, Sy\rangle}_A$ holds for every $x, y\in \mathcal{H}$, that is, $AS = T^*A$ (see \cite{acg1}). In general, the existence of an $A$-adjoint operator is not guaranteed. The set of all operators that admit $A$-adjoints will be denoted by $\mathbb{B}_{A}(\mathcal{H})$. By applying Douglas' theorem \cite{doug}, we get
$$\mathbb{B}_{A}(\mathcal{H})=\left\{T\in \mathbb{B}(\mathcal{H});\;\mathcal{R}(T^{*}A)\subseteq \mathcal{R}(A)\right\}.$$
If $T\in\mathbb{B}_{A}(\mathcal{H})$, then the ``reduced" solution of the equation $AX = T^*A$ is a distinguished $A$-adjoint
operator of $T$, which will be denoted by $T^{\sharp_A}$. We remark that if $T\in\mathbb{B}_{A}(\mathcal{H})$, then $T^{\sharp_A} \in \mathbb{B}_A({\mathcal{H}})$, $(T^{\sharp_A})^{\sharp_A}=P_{\overline{\mathcal{R}(A)}}TP_{\overline{\mathcal{R}(A)}}$ and $((T^{\sharp_A})^{\sharp_A})^{\sharp_A}=T^{\sharp_A}$. Moreover, if $S\in \mathbb{B}_A(\mathcal{H})$, then $TS \in\mathbb{B}_A({\mathcal{H}})$ and $(TS)^{\sharp_A}=S^{\sharp_A}T^{\sharp_A}.$ An operator $T\in\mathbb{B}(\mathcal{H})$ is said to be $A$-selfadjoint if $AT$ is selfadjoint, that is, $AT = T^*A$. Obviously if $T$ is $A$-selfadjoint, then $T\in\mathbb{B}_{A}(\mathcal{H})$. However, in general, the equality $T = T^{\sharp_A}$ may not hold. More precisely, one can verify that if $T\in\mathbb{B}_{A}(\mathcal{H})$, then $T = T^{\sharp_A}$ if and only if $T$ is $A$-selfadjoint and $\mathcal{R}(T) \subseteq \overline{\mathcal{R}(A)}$. Further, we recall that an operator $T$ is called $A$-positive if $AT\geq0$ and we write $T\geq_{A}0$. Clearly, $A$-positive operators are always $A$-selfadjoint. Now, if we denote by $\mathbb{B}_{A^{1/2}}(\mathcal{H})$ the set of all operators admitting $A^{1/2}$-adjoints, then another application of Douglas' theorem \cite{doug} gives
\begin{align*}
\mathbb{B}_{A^{1/2}}(\mathcal{H}) = \big\{T\in \mathbb{B}(\mathcal{H})\,; \,\, \exists\, c>0\,;
\,\,{\|Tx\|}_{A}\leq c{\|x\|}_{A}, \,\, \forall x\in \mathcal{H}\big\}.
\end{align*}
It is clear that $\mathbb{B}_{A}(\mathcal{H})$ and $\mathbb{B}_{A^{1/2}}(\mathcal{H})$ are two subalgebras of
$\mathbb{B}(\mathcal{H})$ which are neither closed nor dense in $\mathbb{B}(\mathcal{H})$.
Further, the proper inclusions $\mathbb{B}_{A}(\mathcal{H}) \subseteq \mathbb{B}_{A^{1/2}}(\mathcal{H})
 \subseteq \mathbb{B}(\mathcal{H})$ hold with equality if $A$ is injective and has a closed range in $\mathcal{H}$. For more details, we refer the reader to \cite{acg1,acg2,feki01}.

Given $T\in\mathbb{B}(\mathcal{H})$. If there exists $\lambda>0$ such that $\|Tx \|_{A} \leq \lambda \|x\|_{A}$, for all $x\in\overline{\mathcal{R}(A)}$, then it holds:
\begin{equation*}\label{semii}
\|T\|_A:=\sup_{\substack{x\in \overline{\mathcal{R}(A)},\\ x\not=0}}\frac{\|Tx\|_A}{\|x\|_A}=\displaystyle\sup_{\substack{x\in \overline{\mathcal{R}(A)},\\ \|x\|_A= 1}}\|Tx\|_{A}<\infty.
\end{equation*}
Of course, it $A=I$ we reach the definition of the classical operator norm. It was shown in \cite{fga} that for $T\in\mathbb{B}_{A^{1/2}}(\mathcal{H})$, we have
\begin{align}\label{fg}
\|T\|_A
&=\sup\big\{{\|Tx\|}_A\,; \,\,x\in \mathcal{H},\, {\|x\|}_A =1\big\}\nonumber\\
&=\sup\left\{|\langle Tx, y\rangle_A|\,;\;x,y\in \mathcal{H},\,\|x\|_{A}=\|y\|_{A}= 1\right\}.
\end{align}
It is useful to note that for every $T,S\in \mathbb{B}_{A^{1/2}}(\mathcal{H})$, we have
\begin{equation}\label{sousmultiplicative}
\|TS\|_A\leq \|T\|_A\|S\|_A.
\end{equation}
Further, clearly for $T\in\mathbb{B}_{A}(\mathcal{H})$, the operators $T^{\sharp_A}T$ and  $TT^{\sharp_A}$ are $A$-positive. In addition, it was shown in \cite{acg2} that
\begin{align}\label{diez}
{\|T^{\sharp_A}T\|}_A = {\|TT^{\sharp_A}\|}_A = {\|T\|}^2_A = {\|T^{\sharp_A}\|}^2_A.
\end{align}
Recently, the $A$-spectral radius of $A$-bounded operators is introduced by the present author in \cite{feki01} as follows
\begin{equation}\label{newrad}
r_A(T):=\displaystyle\inf_{n\geq 1}\|T^n\|_A^{\frac{1}{n}}=\displaystyle\lim_{n\to\infty}\|T^n\|_A^{\frac{1}{n}}.
\end{equation}
We mention that the second equality in \eqref{newrad} is also proved by the present author in \cite[Theorem 1]{feki01}. In addition, $r_A(\cdot)$ satisfies the commutativity property, which asserts that
\begin{equation}\label{commut}
r_A(TS)=r_A(ST),
\end{equation}
for every $T,S\in \mathbb{B}_{A^{1/2}}(\mathcal{H})$ (see \cite{feki01}). In all that follows, for any arbitrary operator $X\in \mathbb{B}_A({\mathcal H})$, we denote
$$\Re_A(X):=\frac{X+X^{\sharp_A}}{2}\;\;\text{ and }\;\;\Im_A(X):=\frac{X-X^{\sharp_A}}{2i}.$$
For $T\in\mathbb{B}(\mathcal{H})$, the $A$-numerical radius of an operator $T$ was firstly defined by Saddi in \cite{saddi} by
\begin{align*}
\omega_A(T)
&:= \sup \left\{|\langle Tx, x\rangle_A|\,;\;x\in\mathcal{H},\|x\|_A = 1\right\}.
\end{align*}
Recently, this concept received considerable attention by many authors. For more details, we refer the reader to \cite{bakfeki01,BPN,BPN2,PPglma,feki01,feki03,NSD,rout} and the references therein. It should be mentioned here that it may happen that ${\|T\|}_A$ and $\omega_A(T)$ are equal to $+ \infty$ for some $T\in\mathbb{B}(\mathcal{H})$ (see \cite{feki01,feki03}). However, it was shown in \cite{bakfeki01} that the above quantities are equivalent seminorms on $\mathbb{B}_{A^{1/2}}(\mathcal{H})$. More precisely, we have
\begin{equation}\label{refine1}
\frac{1}{2} \|T\|_A\leq\omega_A(T) \leq \|T\|_A.
\end{equation}
Recently, several refinements of the inequalities \eqref{refine1} have been proved by many authors (e.g., see \cite{feki03,BPN}, and the references therein). In particular, it has been shown that for $T\in\mathbb{B}_{A}(\mathcal{H})$, we have
\begin{align}\label{kit}
  \frac{1}{2}\sqrt{\|T^{\sharp_A} T+TT^{\sharp_A}\|_A}\le  \omega_A\left(T\right) \le \frac{\sqrt{2}}{2}\sqrt{\|T^{\sharp_A} T+TT^{\sharp_A}\|_A},
\end{align}
(see \cite{feki03}). If $A=I$, we get the well-known inequalities proved by Kittaneh in \cite[Theorem 1]{FK}. One main target of the present paper is to prove some new refinements of the first inequality in \eqref{kit}. The inspiration of our investigation comes from the recent works by Moradi et al. \cite{hasm,minmo}. Some of the obtained results are new even in the case that the underlying operator $A$ is the identity operator. In particular, among other inequalities, we prove that
for every $T \in \mathbb{B}_A(\mathcal{H})$ we have
\begin{align*}
&\frac{1}{2}\sqrt{\|T^{\sharp_A} T+TT^{\sharp_A}\|_A}\\
&\leq\frac{1}{2}\sqrt{2\omega_A^2(T)+\sqrt{\left(\|\Re_A(T)\|_A^2-\|\Im_A(T)\|_A^2\right)^2+4\|\Re_A(T)\Im_A(T)\|_A^2}}\leq \omega_A(T).
\end{align*}
In addition, several new refinements of the triangle inequality related to $\|\cdot\|_A$ are proved. Mainly, we prove that for every $T,S\in \mathbb{B}_A(\mathcal{H})$ we have
\begin{align*}
\| T+S\|_A
&\leq \sqrt{\frac{1}{2}\left(\|T\|_A^2+\|S\|_A^2+\sqrt{\left(\|T\|_A^2-\|S\|_A^2\right)^2+4\|TS^{\sharp_A}\|_A^2} \right)+2\omega_A(S^{\sharp_A}T)}\nonumber\\
&\leq \|T\|_A+\|S\|_A.
\end{align*}
Several applications of the obtained inequalities are also given.
\section{Results}\label{s2}
In this section, we present our results. In order to achieve the goals of the present section, we need the following lemma.
\begin{lemma}\label{s1}(\cite{bakfeki04,feki03})
Let $T\in \mathbb{B}(\mathcal{H})$ be an $A$-selfadjoint operator. Then, the following assertions hold:
\begin{itemize}
  \item [(i)] $T^{\sharp_A}$ is $A$-selfadjoint and $({T^{\sharp_A}})^{\sharp_A}=T^{\sharp_A}$.
  \item [(ii)] $\|T\|_{A}=\omega_A(T)=r_A(T)$.
  \item [(iii)] $\|T^n\|_{A}=\|T\|_{A}^n$ for any positive integer $n$.
    \item [(vi)] $T^{2n}\geq_A 0$ for any positive integer $n$.
\end{itemize}
\end{lemma}
Our first result in this provides a refinement of the first inequality in \eqref{kit} and reads as follows.
\begin{theorem}\label{thm1}
Let $T \in \mathbb{B}_A(\mathcal{H})$. Then
\begin{equation}\label{asli2}
\frac{1}{2}\sqrt{\|T^{\sharp_A} T+TT^{\sharp_A}\|_A}\leq\frac{\sqrt{2}}{2}\sqrt{\left\|\Re_A(T)\right\|_A^2+\left\|\Im_A(T)\right\|_A^2}\le  \omega_A\left(T\right).
\end{equation}
\end{theorem}
\begin{proof}
Observe first that $T$ can be decomposed as $T=\Re_A(T)+i\Im_A(T)$. Further, it is not difficult to verify that $\Re_A(T)$ and $\Im_A(T)$ are $A$-selfadjoint operators. Thus by applying Lemma \ref{s1} (i) we get
$$({[\Re_A(T)]^{\sharp_A}})^{\sharp_A}=[\Re_A(T)]^{\sharp_A}\;\text{ and }\; ({[\Im_A(T)]^{\sharp_A}})^{\sharp_A}=[\Im_A(T)]^{\sharp_A}.$$
So, a short calculations shows that
\begin{align}\label{sharpp}
\frac{1}{2}\left(TT^{\sharp_A} + T^{\sharp_A} T\right)^{\sharp_A}
&=\left([\Re_A(T)]^{\sharp_A}\right)^2+\left([\Im_A(T)]^{\sharp_A}\right)^2.
\end{align}
Thus, one observes that
\begin{align*}
\frac{1}{4}\|T^{\sharp_A} T+TT^{\sharp_A}\|_A
&=\frac{1}{4}\left\|\left(TT^{\sharp_A} + T^{\sharp_A} T\right)^{\sharp_A}\right\|_A\\
&=\frac{1}{2}\left\|\left([\Re_A(T)]^{\sharp_A}\right)^2+\left([\Im_A(T)]^{\sharp_A}\right)^2\right\|_A\\
&\leq\frac{1}{2}\left\|\left([\Re_A(T)]^{\sharp_A}\right)^2\right\|_A+\frac{1}{2}\left\|\left([\Im_A(T)]^{\sharp_A}\right)^2\right\|_A\\
&\leq\frac{1}{2}\left\|\Re_A^2(T)\right\|_A+\frac{1}{2}\left\|\Im_A^2(T)\right\|_A,
\end{align*}
where the last inequality follows from the fact that $\|X^{\sharp_A}\|_A=\|X\|_A$ for all $X\in \mathbb{B}_A(\mathcal{H})$. So, by applying Lemma \ref{s1} (iii), we get
\begin{align}\label{cc3}
\frac{1}{4}\|T^{\sharp_A} T+TT^{\sharp_A}\|_A\leq\frac{1}{2}\left\|\Re_A(T)\right\|_A^2+\frac{1}{2}\left\|\Im_A(T)\right\|_A^2.
\end{align}
On the other hand, let $x\in \mathcal{H}$ be such that $\|x\|_A=1$. Then, we verify that
\begin{align*}
\big|{\langle Tx, x \rangle}_A\big|^2
&= \big|{\langle \left[\Re_A(T)+i\Im_A(T)\right]x, x \rangle}_A\big|^2\\
&= \big|{\langle \Re_A(T)x, x \rangle}_A+i{\langle \Im_A(T)x, x \rangle}_A\big|^2\\
& = \big|{\langle\Re_A(T)x, x\rangle}_A\big|^2 + \big|{\langle \Im_A(T)x, x\rangle}_A\big|^2\\
&\geq\big|{\langle\Re_A(T)x, x\rangle}_A\big|^2.
\end{align*}
So, by taking the supremum over all $x\in \mathcal{H}$ with $\|x\|_A=1$ in the above inequality and then using Lemma \ref{s1} (ii), we obtain
\begin{equation}\label{cc1}
\|\Re_A(T)\|_A^2\leq\omega_A^2(T).
\end{equation}
Similarly, one can prove that
\begin{equation}\label{cc2}
\|\Im_A(T)\|_A^2\leq\omega_A^2(T).
\end{equation}
By combining \eqref{cc3} together with \eqref{cc1} and \eqref{cc2}, we get the desired result. This finishes the proof of the theorem.
\end{proof}

In order to derive  a new improvement of the first inequality in \eqref{kit}, we need the following lemma.
\begin{lemma}\label{immmmm2}
Let $T,S\in \mathbb{B}(\mathcal{H})$ be $A$-selfadjoint operators. Then,
\begin{equation*}
\| T^2+S^2\|_A\leq \frac{1}{2}\left(\|T^2\|_A+\|S^2\|_A+\sqrt{\left(\|T^2\|_A-\|S^2\|_A\right)^2+4\|TS\|_A^2} \right).
\end{equation*}
\end{lemma}
\begin{proof}
Since, $T,S\in \mathbb{B}(\mathcal{H})$ are $A$-selfadjoint operators, then by applying Lemma \ref{s1} (vi) one see that $T^2+S^2\geq_A0$. So, by Lemma \ref{s1} (ii) we have
\begin{equation}\label{jj27}
\| T^2+S^2\|_A =r_A\left(T^2+S^2\right).
\end{equation}
On the other hand, it can be checked that
\begin{align*}
r_A(T^2+S^2)
&=r_{\mathbb{A}}\left[\begin{pmatrix}
T^2+S^2&0 \\
0&0
\end{pmatrix}\right]=r_{\mathbb{A}}\left[\begin{pmatrix}
T&S \\
0&0
\end{pmatrix}\begin{pmatrix}
T&0 \\
S&0
\end{pmatrix}\right],
\end{align*}
where $\mathbb{A}=\begin{pmatrix}
A &0\\
0 &A
\end{pmatrix}\in \mathbb{B}(\mathcal{H}\oplus \mathcal{H})$ is the $2\times 2$ positive diagonal operator
matrix whose each diagonal entry is the positive operator $A$. Further, by using \eqref{commut} we get
\begin{align}\label{jj285}
r_A(T^2+S^2)
&=r_{\mathbb{A}}\left[\begin{pmatrix}
T&0 \\
S&0
\end{pmatrix}\begin{pmatrix}
T&S \\
0&0
\end{pmatrix}\right]\nonumber\\
&=r_{\mathbb{A}}\left[\begin{pmatrix}
T^2&TS\\
ST&S^2
\end{pmatrix}\right].
\end{align}
Further, by \cite{fmjom} we have
\begin{equation}\label{jj29}
r_{\mathbb{A}}\left[\begin{pmatrix}
T^2&TS\\
ST&S^2
\end{pmatrix}\right]\leq r\left[\begin{pmatrix}
\|T^2\|_A&\|TS\|_A\\
\|ST\|_A&\|S^2\|_A
\end{pmatrix}\right].
\end{equation}
Therefore, by applying \eqref{jj27} together with \eqref{jj285} and \eqref{jj29} we see that
\begin{align}\label{dd1}
\|T^2+S^2\|_A
&\leq r\left[\begin{pmatrix}
\|T^2\|_A&\|TS\|_A\\
\|ST\|_A&\|S^2\|_A
\end{pmatrix}\right].
\end{align}
Since $T\geq_A0$ and $S\geq_A0$, then $T$ and $S$ are $A$-selfadjoint. This implies, through Lemma \ref{s1} (i) that $T^{\sharp_A}$ and $S^{\sharp_A}$ are also $A$-selfadjoint. Thus, Lemma \ref{s1} (i) gives $(T^{\sharp_A})^{\sharp_A}=T^{\sharp_A}$ and $(S^{\sharp_A})^{\sharp_A}=S^{\sharp_A}$. So, we obtain
$$\|TS\|_A=\|(TS)^{\sharp_A}\|_A=\|S^{\sharp_A}T^{\sharp_A}\|_A=\|(T^{\sharp_A})^{\sharp_A}(S^{\sharp_A})^{\sharp_A}\|_A=\|T^{\sharp_A}S^{\sharp_A}\|_A=\|ST\|_A.$$
Hence, $\begin{pmatrix}
\|T^2\|_A&\|TS\|_A\\
\|ST\|_A&\|S^2\|_A
\end{pmatrix}$ is a symmetric matrix. This yields that
\begin{align*}
 r\left[\begin{pmatrix}
\|T^2\|_A&\|TS\|_A\\
\|ST\|_A&\|S^2\|_A
\end{pmatrix}\right]
&=\frac{1}{2}\left(\|T^2\|_A+\|S^2\|_A+\sqrt{\left(\|T^2\|_A-\|S^2\|_A\right)^2+4\|TS\|_A^2} \right).
\end{align*}
This finishes the proof by taking into consideration \eqref{dd1}.
\end{proof}
Now, we are in a position to prove one of our main results in this paper.
\begin{theorem}\label{thn}
Let $T \in \mathbb{B}_A(\mathcal{H})$. Then
\begin{align*}
&\frac{1}{4}\|T^{\sharp_A} T+TT^{\sharp_A}\|_A\\
&\leq\frac{1}{4}\left(2\omega_A^2(T)+\sqrt{\left(\|\Re_A(T)\|_A^2-\|\Im_A(T)\|_A^2\right)^2+4\|\Re_A(T)\Im_A(T)\|_A^2} \right)\\
&\leq \omega_A^2(T).
\end{align*}
\end{theorem}
\begin{proof}
Notice first that $T$ can be written as $T=[\Re_A(T)]+i[\Im_A(T)]$. By using an argument similar to that used in proof of Theorem \ref{thm1}, we get
\begin{align*}
\frac{1}{4}\|T^{\sharp_A} T+TT^{\sharp_A}\|_A
&=\frac{1}{2}\left\|\left([\Re_A(T)]^{\sharp_A}\right)^2+\left([\Im_A(T)]^{\sharp_A}\right)^2\right\|_A\\
&=\frac{1}{2}\left\|\Re_A^2(T)+\Im_A^2(T)\right\|_A,
\end{align*}
Since the operators $\Re_A(T)$ and $\Im_A(T)$ are $A$-selfadjoint, then an application of Lemma \ref{immmmm2} gives
\begin{align*}
&\frac{1}{4}\left\|T^{\sharp_A} T+TT^{\sharp_A}\right\|_A\\
&\leq \frac{1}{4}\left(\|\Re_A^2(T)\|_A+\|\Im_A^2(T)\|_A+\sqrt{\left(\|\Re_A^2(T)\|_A-\|\Im_A^2(T)\|_A\right)^2+4\|\Re_A(T)\Im_A(T)\|_A^2} \right)\\
&= \frac{1}{4}\left(\|\Re_A(T)\|_A^2+\|\Im_A(T)\|_A^2+\sqrt{\left(\|\Re_A(T)\|_A^2-\|\Im_A(T)\|_A^2\right)^2+4\|\Re_A(T)\Im_A(T)\|_A^2} \right),
\end{align*}
where the last equality follows by applying Lemma \ref{s1} (iii). Further, by applying \eqref{cc1} and \eqref{cc2} we see that
\begin{align*}
&\frac{1}{4}\left\|T^{\sharp_A} T+TT^{\sharp_A}\right\|_A\\
&\leq\frac{1}{4}\left(2\omega_A^2(T)+\sqrt{\left(\|\Re_A(T)\|_A^2-\|\Im_A(T)\|_A^2\right)^2+4\|\Re_A(T)\Im_A(T)\|_A^2} \right).
\end{align*}
This proves the first inequality in Theorem \ref{thn}. Now, by using \eqref{sousmultiplicative} and then making simple calculations we get
\begin{align*}
&\frac{1}{4}\left(2\omega_A^2(T)+\sqrt{\left(\|\Re_A(T)\|_A^2-\|\Im_A(T)\|_A^2\right)^2+4\|\Re_A(T)\Im_A(T)\|_A^2} \right)\\
&\leq \frac{1}{4}\left(2\omega_A^2(T)+\sqrt{\left(\|\Re_A(T)\|_A^2-\|\Im_A(T)\|_A^2\right)^2+4\|\Re_A(T)\|_A^2\|\Im_A(T)\|_A^2} \right)\\
&=\frac{1}{4}\left(2\omega_A^2(T)+\sqrt{\left(\|\Re_A(T)\|_A^2+\|\Im_A(T)\|_A^2\right)^2} \right)\\
&=\frac{1}{4}\left(2\omega_A^2(T)+\|\Re_A(T)\|_A^2+\|\Im_A(T)\|_A^2 \right)\\
&\leq\frac{1}{4}\left(2\omega_A^2(T)+2\omega_A^2(T)\right)=\omega_A^2(T).
\end{align*}
This finished the proof of the theorem.
\end{proof}
The following lemma plays a crucial role in proving our next result and provides a new refinement of the triangle inequality related to $\|\cdot\|_A$.
\begin{lemma}\label{int}
Let $T\in\mathbb{B}_{A^{1/2}}(\mathcal{H})$. Then
\begin{align*}
\|T+S\|_A
&\le\int_{0}^1\left\|\lambda T+(1-\lambda)\frac{T+S}{2}\right\|_Ad\lambda+\int_{0}^1\left\|\lambda S+(1-\lambda)\frac{T+S}{2}\right\|_Ad\lambda\\
&\le \|T\|_A+\|S\|_A.
\end{align*}
\end{lemma}
\begin{proof}
Let $\lambda\in [0,1]$. Then, one observes that
\begin{align*}
\|T+S\|_A
&=\left\|\lambda T+(1-\lambda)\frac{T+S}{2}+\lambda S+(1-\lambda)\frac{T+S}{2}\right\|_A\\
&\le \left\|\lambda T+(1-\lambda)\frac{T+S}{2}\right\|_A+\left\|\lambda S+(1-\lambda)\frac{T+S}{2}\right\|_A.
\end{align*}
This proves the first inequality in Lemma \ref{int} by taking integral over $\lambda\in [0,1]$. The second inequality in Lemma \ref{int} follows immediately by applying the triangle inequality related to $\|\cdot\|_A$ and then making simple calculations.
\end{proof}
Another improvement of the first inequality in \eqref{kit} can be stated as follows.
\begin{theorem}\label{thm3}
Let $T\in\mathbb{B}_A(\mathcal{H})$. Then
\begin{align*}
\frac{1}{4}\|T^{\sharp_A} T+TT^{\sharp_A}\|_A
&\le \frac{1}{2}\int_{0}^1\left\|\lambda \Re_A^2(T)+(1-\lambda)\frac{\Re_A^2(T)+\Im_A^2(T)}{2}\right\|_Ad\lambda\\
&\quad\quad+\frac{1}{2}\int_{0}^1\left\|\lambda \Im_A^2(T)+(1-\lambda)\frac{\Re_A^2(T)+\Im_A^2(T)}{2}\right\|_Ad\lambda\\
&\le \omega_A^{2}\left( T\right).
\end{align*}
\end{theorem}
\begin{proof}
Since $T=\Re_A(T)+i\Im_A(T)$, then by using an argument similar to that used in proof of Theorem \ref{thm1}, we get
\begin{align*}
\frac{1}{4}\|T^{\sharp_A} T+TT^{\sharp_A}\|_A
&=\frac{1}{2}\left\|\left([\Re_A(T)]^{\sharp_A}\right)^2+\left([\Im_A(T)]^{\sharp_A}\right)^2\right\|_A\\
&=\frac{1}{2}\left\|[\Re_A^2(T)]^{\sharp_A}+[\Im_A^2(T)]^{\sharp_A}\right\|_A\\
&=\frac{1}{2}\left\|\Re_A^2(T)+\Im_A^2(T)\right\|_A.
\end{align*}
So, by applying Lemma \ref{int}, we see that
\begin{align*}
\frac{1}{4}\|T^{\sharp_A} T+TT^{\sharp_A}\|_A
&\le \frac{1}{2}\int_{0}^1\left\|\lambda \Re_A^2(T)+(1-\lambda)\frac{\Re_A^2(T)+\Im_A^2(T)}{2}\right\|_Ad\lambda\\
&\quad\quad+\frac{1}{2}\int_{0}^1\left\|\lambda \Im_A^2(T)+(1-\lambda)\frac{\Re_A^2(T)+\Im_A^2(T)}{2}\right\|_Ad\lambda\\
&\leq\frac{1}{2}\left(\|\Re_A^2(T)\|_A+\|\Im_A^2(T)\|_A\right)\\
&=\frac{1}{2}\left(\|\Re_A(T)\|_A^2+\|\Im_A(T)\|_A^2\right)\quad(\text{by Lemma } \ref{s1} \text{(iii)})\\
&\leq \omega_A^{2}\left( T\right),
\end{align*}
where the last inequality follows by applying \eqref{cc1} together with \eqref{cc2}. Hence the proof is complete.
\end{proof}

In the following theorem, we prove another refinement of the triangle inequality related to $\|\cdot\|_A$.
\begin{theorem}\label{refan}
Let $T,S\in\mathbb{B}_{A}(\mathcal{H})$. Then
\begin{equation}
\left\| T\pm S\right\|_A\leq\sqrt{\|T^{\sharp_A}T+S^{\sharp_A}S\|_A+\|T^{\sharp_A}S+S^{\sharp_A}T\|_A}\leq \|T\|_A+\|S\|_A.  \label{2.7tw}
\end{equation}
\end{theorem}
\begin{proof}
Let $x\in \mathcal{H}$ be such that $\|x\|_A=1$. Then we see that
\begin{align*}
\left\| Tx+Sx\right\|_A ^{2}
& =\left\| Tx\right\|_A ^{2}+\left\langle Tx, Sx\right\rangle_A+\left\langle Sx, Tx\right\rangle_A +\left\| Sx\right\|_A ^{2} \\
&\leq \left\langle \left( T^{\sharp_A}T+S^{\sharp_A}S\right) x, x\right\rangle_A
+\left|\left\langle \left(T^{\sharp_A}S+S^{\sharp_A}T\right) x, x\right\rangle_A\right|.
\end{align*}
Further, by applying the Cauchy-Schwarz inequality we get
\begin{align*}
\left\| Tx+Sx\right\|_A ^{2}
& \leq \|T^{\sharp_A}T+S^{\sharp_A}S\|_A+\|T^{\sharp_A}S+S^{\sharp_A}T\|_A.
\end{align*}%
 So, by taking the supremum over all $x\in \mathcal{H}$ with $\left\| x\right\|_A =1$ in the above inequality we get%
\begin{align*}
\left\|T+S\right\|_A ^{2}
&\leq \|T^{\sharp_A}T+S^{\sharp_A}S\|_A+\|T^{\sharp_A}S+S^{\sharp_A}T\|_A.
\end{align*}%
Similarly, we show that
\begin{align*}
\left\|T-S\right\|_A ^{2}
&\leq \|T^{\sharp_A}T+S^{\sharp_A}S\|_A+\|T^{\sharp_A}S+S^{\sharp_A}T\|_A.
\end{align*}%
Hence, we get the first inequality in \eqref{2.7tw}. Moreover, by applying the triangle inequality together with \eqref{sousmultiplicative} and \eqref{diez} we see that
\begin{align*}
\left\| T\pm S\right\|_A
&\leq\sqrt{\|T^{\sharp_A}T+S^{\sharp_A}S\|_A+\|T^{\sharp_A}S+S^{\sharp_A}T\|_A}\\
&\leq \sqrt{\|T\|_A^2+\|S\|_A^2+2\|T\|_A\|S\|_A}\\
&=\sqrt{\left(\|T\|_A+\|S\|_A\right)^2}\\
&= \|T\|_A+\|S\|_A.
\end{align*}%
This immediately proves the required result.
\end{proof}
\begin{remark}
We note that the inequalities obtained in Theorem \ref{refan} cover and refine the recent inequalities due to Bhunia et al. (see \cite[Theorem 2.4]{parkiv}.
\end{remark}
As an application of Theorem \ref{refan}, we derive another improvement of the first inequality in \eqref{kit}.
\begin{corollary}\label{corf}
Let $T\in\mathbb{B}_A(\mathcal{H})$. Then
\begin{equation*}
\frac{1}{4}\|T^{\sharp_A} T+TT^{\sharp_A}\|_A\leq\frac{\sqrt{2}}{2}\sqrt{\omega_A^4(T)+\big\|\Im_A^2(T)\Re_A^2(T)\big\|_A}\le \omega_A^{2}(T).
\end{equation*}
\end{corollary}
\begin{proof}
Since $T=\Re_A(T)+i\Im_A(T)$, then by using \eqref{sharpp} we observe that
\begin{align*}
\frac{1}{4}\|T^{\sharp_A} T+TT^{\sharp_A}\|_A
&=\frac{1}{2}\left\|\left[\Re_A^2(T)\right]^{\sharp_A}+\left[\Im_A^2(T)\right]^{\sharp_A}\right\|_A.
\end{align*}
Moreover, in view of Lemma \ref{s1} (vi), the operators $\Re_A^2(T)$ and $\Im_A^2(T)$ are $A$-positive and
\begin{equation}\label{1j}
\left[\left[\Re_A^2(T)\right]^{\sharp_A}\right]^{\sharp_A}=\left[\Re_A^2(T)\right]^{\sharp_A}\;\;\text{ and }\;\; \left[\left[\Im_A^2(T)\right]^{\sharp_A}\right]^{\sharp_A}=\left[\Im_A^2(T)\right]^{\sharp_A}.
\end{equation}
So, by using Theorem \ref{refan} together with \eqref{1j} we observe that
\begin{align*}
&\frac{1}{16}\|T^{\sharp_A} T+TT^{\sharp_A}\|_A^2\\
&\leq\frac{1}{4}\big\|\big(\left[\Re_A^2(T)\right]^{\sharp_A}\big)^2+\big(\left[\Im_A^2(T)\right]^{\sharp_A}\big)^2\big\|_A\\
&\quad\quad\quad\quad+\frac{1}{4}\big\|\left[\Re_A^2(T)\right]^{\sharp_A}\left[\Im_A^2(T)\right]^{\sharp_A}+\left[\Im_A^2(T)\right]^{\sharp_A}\left[\Re_A^2(T)\right]^{\sharp_A}\big\|_A\\
&\leq\frac{1}{4}\big\|\big(\left[\Re_A^2(T)\right]^{\sharp_A}\big)^2+\big(\left[\Im_A^2(T)\right]^{\sharp_A}\big)^2\big\|_A\\
&\quad\quad\quad\quad+\frac{1}{4}\big\|\left[\Re_A^2(T)\right]^{\sharp_A}\left[\Im_A^2(T)\right]^{\sharp_A}\big\|_A+\frac{1}{4}\big\|\left[\Im_A^2(T)\right]^{\sharp_A}\left[\Re_A^2(T)\right]^{\sharp_A}\big\|_A.
\end{align*}
On the other hand, since $\|X\|_A=\|X^{\sharp_A}\|_A$ for all $X\in\mathbb{B}_A(\mathcal{H})$, then by \eqref{1j} we have
\begin{align*}
\big\|\left[\Re_A^2(T)\right]^{\sharp_A}\left[\Im_A^2(T)\right]^{\sharp_A}\big\|_A
&=\left\|\left[\left[\Im_A^2(T)\right]^{\sharp_A}\right]^{\sharp_A}\left[\left[\Re_A^2(T)\right]^{\sharp_A}\right]^{\sharp_A}\right\|_A\\
&=\big\|\left[\Im_A^2(T)\right]^{\sharp_A}\left[\Re_A^2(T)\right]^{\sharp_A}\big\|_A.
\end{align*}
Hence, we obtain
\begin{align*}
&\frac{1}{16}\|T^{\sharp_A} T+TT^{\sharp_A}\|_A^2\\
&\leq\frac{1}{4}\big\|\big(\left[\Re_A^2(T)\right]^{\sharp_A}\big)^2+\big(\left[\Im_A^2(T)\right]^{\sharp_A}\big)^2\big\|_A+\frac{1}{2}\big\|\left[\Re_A^2(T)\right]^{\sharp_A}\left[\Im_A^2(T)\right]^{\sharp_A}\big\|_A\\
&\leq\frac{1}{4}\big\|\Re_A^4(T)+\Im_A^4(T)\big\|_A+\frac{1}{2}\big\|\Im_A^2(T)\Re_A^2(T)\big\|_A.
\end{align*}
This implies that
\begin{align*}
\frac{1}{4}\|T^{\sharp_A} T+TT^{\sharp_A}\|_A
&\leq\frac{1}{2}\sqrt{\big\|\Re_A^4(T)+\Im_A^4(T)\big\|_A+2\big\|\Im_A^2(T)\Re_A^2(T)\big\|_A}\\
&\leq\frac{1}{2}\sqrt{\big\|\Re_A^4(T)\big\|_A+\big\|\Im_A^4(T)\big\|_A+2\big\|\Im_A^2(T)\Re_A^2(T)\big\|_A}\\
&=\frac{1}{2}\sqrt{\big\|\Re_A(T)\big\|_A^4+\big\|\Im_A(T)\big\|_A^4+2\big\|\Im_A^2(T)\Re_A^2(T)\big\|_A},
\end{align*}
where the last equality follows from Lemma \ref{s1} (iii) since $\Re_A(T)$ and $\Im_A(T)$ are $A$-selfadjoint operators. By using \eqref{cc1} and \eqref{cc2} we obtain
\begin{align*}
\frac{1}{4}\|T^{\sharp_A} T+TT^{\sharp_A}\|_A
&\leq\frac{\sqrt{2}}{2}\sqrt{\omega_A^4(T)+\big\|\Im_A^2(T)\Re_A^2(T)\big\|_A}.
\end{align*}
This shows the first inequality in Corollary \ref{corf}. Finally, by using \eqref{sousmultiplicative} and similar arguments as above, we see that
\begin{align*}
\frac{\sqrt{2}}{2}\sqrt{\omega_A^4(T)+\big\|\Im_A^2(T)\Re_A^2(T)\big\|_A}
&\leq\frac{\sqrt{2}}{2}\sqrt{\omega_A^4(T)+\big\|\Re_A(T)\big\|_A^2\big\|\Im_A(T)\big\|_A^2}\\
&\leq\frac{\sqrt{2}}{2}\sqrt{\omega_A^4(T)+\omega_A^4(T)}=\omega_A^2(T).
\end{align*}
Hence, the proof is complete.
\end{proof}

In order to prove our second main result in this paper, we need the following paper.
\begin{lemma}\label{drog0}
Let $T,S\in \mathbb{B}_A(\mathcal{H})$. Then
\[\left\|T\pm S\right\|_A\le \sqrt{\left\|TT^{\sharp_A}+SS^{\sharp_A}\right\|_A+2\omega_A(TS^{\sharp_A})}\leq \|T\|_A+\|S\|_A.\]
\end{lemma}
\begin{proof}
It was shown in \cite[Lemma 2.14.]{fhac} that for $T,S\in\mathbb{B}_{A}(\mathcal{H})$ we have
\begin{equation*}
\max\Big\{\left\| T+S\right\|_A^{2}, \left\| T-S\right\|_A^{2}\Big\}-\|TT^{\sharp_A}+SS^{\sharp_A}\|_A\leq 2\omega_A\left(TS^{\sharp_A}\right).
\end{equation*}
This immediately proves the first inequality in Lemma \ref{drog0}. The second inequality in Lemma \ref{drog0} can be seen easily by remarking that $\omega_A(TS^{\sharp_A})\leq \|TS^{\sharp_A}\|_A$ and then proceeding as in the proof of Theorem \ref{refan}.

This finishes the proof of the desired result.
\end{proof}
Another improvement of the first inequality in \eqref{kit}, that involves $\Re_A(T)$ and $\Im_A(T)$ can be seen as follows.
\begin{theorem}\label{thm2}
Let $T\in\mathbb{B}_A(\mathcal{H})$. Then
\begin{equation*}
\frac{1}{4}\|T^{\sharp_A} T+TT^{\sharp_A}\|_A\le \frac{\sqrt{2}}{2}\sqrt{\omega_A^4(T)+\omega_A\Big(\Im_A^2(T)\Re_A^2(T)\Big)}\le \omega_A^{2}(T).
\end{equation*}
\end{theorem}
\begin{proof}
Since $T=\Re_A(T)+i\Im_A(T)$. Then, by using an argument similar to that used in proof of Theorem \ref{thm1}, we get
\begin{align*}
\frac{1}{4}\|T^{\sharp_A} T+TT^{\sharp_A}\|_A
&=\frac{1}{4}\left\|\left(TT^{\sharp_A} + T^{\sharp_A} T\right)^{\sharp_A}\right\|_A\\
&=\frac{1}{2}\left\|\left([\Re_A(T)]^{\sharp_A}\right)^2+\left([\Im_A(T)]^{\sharp_A}\right)^2\right\|_A\\
&=\frac{1}{2}\left\|\left[\Re_A^2(T)\right]^{\sharp_A}+\left[\Im_A^2(T)\right]^{\sharp_A}\right\|_A.
\end{align*}
So, by applying Lemma \ref{drog0} together with \eqref{1j} and then using similar arguments as above, we get
\begin{align*}
&\frac{1}{4}\|T^{\sharp_A} T+TT^{\sharp_A}\|_A\\
&\le \frac{1}{2}\sqrt{\left\|\big(\left[\Re_A^2(T)\right]^{\sharp_A}\big)^2+\big(\left[\Im_A^2(T)\right]^{\sharp_A}\big)^2\right\|_A
+2\omega_A\Big(\left[\Re_A^2(T)\right]^{\sharp_A}\left[\Im_A^2(T)\right]^{\sharp_A}\Big)}\\
&\le \frac{1}{2}\sqrt{\left\|\Re_A^2(T)\right\|^2+\left\|\Im_A^2(T)\right\|^2+2\omega_A\Big(\Im_A^2(T)\Re_A^2(T)\Big)}\\
&\le \frac{1}{2}\sqrt{\left\|\Re_A(T)\right\|_A^4+\left\|\Im_A(T)\right\|_A^4+2\omega_A\Big(\Im_A^2(T)\Re_A^2(T)\Big)}\\
&\le \frac{\sqrt{2}}{2}\sqrt{\omega_A^4(T)+\omega_A\Big(\Im_A^2(T)\Re_A^2(T)\Big)},
\end{align*}
where the last inequality follow by applying \eqref{cc1} together with \eqref{cc2}. Now, we will prove the second inequality in Theorem \ref{thm2}. By using the second inequality in \eqref{refine1}, we see that
\begin{align*}
\frac{\sqrt{2}}{2}\sqrt{\omega_A^4(T)+\omega_A\Big(\Im_A^2(T)\Re_A^2(T)\Big)}
&\le\frac{\sqrt{2}}{2}\sqrt{\omega_A^4(T)+\Big\|\Im_A^2(T)\Re_A^2(T)\Big\|_A}\\
&\le \frac{\sqrt{2}}{2}\sqrt{\omega_A^4(T)+\left\|\Re_A(T)\right\|_A^2\left\|\Im_A(T)\right\|_A^2\Big) }\\
&\le \frac{\sqrt{2}}{2}\sqrt{\omega_A^4(T)+\omega_A^4(T) }= \omega_A^2(T).
\end{align*}
Hence, the proof is complete.
\end{proof}
Another refinement of the triangle inequality related to $\|\cdot\|_A$ can be stated as follows.
\begin{theorem}\label{moradi}
Let $T,S\in \mathbb{B}_A(\mathcal{H})$. Then,
\begin{align*}
\| T+S\|_A
&\leq \sqrt{\frac{1}{2}\left(\|T\|_A^2+\|S\|_A^2+\sqrt{\left(\|T\|_A^2-\|S\|_A^2\right)^2+4\|TS^{\sharp_A}\|_A^2} \right)+2\omega_A(S^{\sharp_A}T)}\nonumber\\
&\leq \|T\|_A+\|S\|_A.
\end{align*}
\end{theorem}
\begin{proof}
Let $x\in \mathcal{H}$ be such that $\|x\|_A=1$. Then, we see that
\begin{align*}
\|(T+S)x\|_A^2
&=\|Tx\|_A^2+\|Sx\|_A^2+2|\langle Tx,Sx\rangle_A|\\
&=\langle T^{\sharp_A}Tx,x\rangle_A+\langle S^{\sharp_A}Sx,x\rangle_A+2|\langle S^{\sharp_A}Tx,x\rangle_A|\\
&\leq\langle \big(T^{\sharp_A}T+S^{\sharp_A}S\big)x,x\rangle_A+2\omega_A(S^{\sharp_A}T)\\
&=\left\|T^{\sharp_A}T+S^{\sharp_A}S\right\|_A+2\omega_A(S^{\sharp_A}T),
\end{align*}
where the last equality follows from Lemma \ref{s1} (ii) since $T^{\sharp_A}T+S^{\sharp_A}S\geq_A0$. Hence,
\begin{align*}
\|(T+S)x\|_A^2\leq\left\|T^{\sharp_A}T+S^{\sharp_A}S\right\|_A+2\omega_A(S^{\sharp_A}T).
\end{align*}
So, by taking the supremum over all $x\in \mathcal{H}$ with $\|x\|_A=1$ in the last inequality we get
\begin{align}\label{21.12}
\|T+S\|_A^2\leq\left\|T^{\sharp_A}T+S^{\sharp_A}S\right\|_A+2\omega_A(S^{\sharp_A}T).
\end{align}
On the other hand, let $\mathbb{A}=\begin{pmatrix}
A &0\\
0 &A
\end{pmatrix}$. Since $T^{\sharp_A}T+S^{\sharp_A}S\geq_A0$, then by applying Lemma \ref{s1} (ii) we observe that
\begin{align*}
\|T^{\sharp_A}T+S^{\sharp_A}S\|_A
& =r_A(T^{\sharp_A}T+S^{\sharp_A}S)\nonumber\\
&=r_{\mathbb{A}}\left[\begin{pmatrix}
T^{\sharp_A}T+S^{\sharp_A}S&0 \\
0&0
\end{pmatrix}\right]\\
&=r_{\mathbb{A}}\left[\begin{pmatrix}
T^{\sharp_A}&S^{\sharp_A} \\
0&0
\end{pmatrix}\begin{pmatrix}
T&0 \\
S&0
\end{pmatrix}\right].
\end{align*}
Moreover, an application of \eqref{commut} gives
\begin{align}\label{jj28}
\|T^{\sharp_A}T+S^{\sharp_A}S\|_A
&=r_{\mathbb{A}}\left[\begin{pmatrix}
T&0 \\
S&0
\end{pmatrix}\begin{pmatrix}
T^{\sharp_A}&S^{\sharp_A} \\
0&0
\end{pmatrix}\right]\nonumber\\
&=r_{\mathbb{A}}\left[\begin{pmatrix}
TT^{\sharp_A}&TS^{\sharp_A}\\
ST^{\sharp_A}&SS^{\sharp_A}
\end{pmatrix}\right]\nonumber\\
&\leq
r\left[\begin{pmatrix}
\|T\|_A^2&\|TS^{\sharp_A}\|_A\\
\|ST^{\sharp_A}\|_A&\|S\|_A^2
\end{pmatrix}\right],
\end{align}
where the last inequality follows from \cite{fmjom} together with \eqref{diez}. In addition, since $\mathcal{R}(T^{\sharp_A})\subseteq \overline{\mathcal{R}(A)}$ then $P_{\overline{\mathcal{R}(A)}}T^{\sharp_A}=T^{\sharp_A}$. So, by applying \eqref{fg} we see that
\begin{align*}
\|TS^{\sharp_A}\|_A
&=\|P_{\overline{\mathcal{R}(A)}}SP_{\overline{\mathcal{R}(A)}}T^{\sharp_A}\|_A\\
&=\|P_{\overline{\mathcal{R}(A)}}ST^{\sharp_A}\|_A\\
&=\sup\left\{|\langle P_{\overline{\mathcal{R}(A)}}ST^{\sharp_A}x, y\rangle_A|\,;\;x,y\in \mathcal{H},\,\|x\|_{A}=\|y\|_{A}= 1\right\}=\|ST^{\sharp_A}\|_A,
\end{align*}
where the last equality holds since $AP_{\overline{\mathcal{R}(A)}}=A$. So, it is not difficult to verify that
\begin{align*}
r\left[\begin{pmatrix}
\|T\|_A^2&\|TS^{\sharp_A}\|_A\\
\|TS^{\sharp_A}\|_A&\|S\|_A^2
\end{pmatrix}\right]
&=\frac{1}{2}\left(\|T\|_A^2+\|S\|_A^2+\sqrt{\left(\|T\|_A^2-\|S\|_A^2\right)^2+4\|TS^{\sharp_A}\|_A^2} \right).
\end{align*}
This implies, through \eqref{jj28}, that
\begin{align*}\label{21.122}
\|T^{\sharp_A}T+S^{\sharp_A}S\|_A\leq\frac{1}{2}\left(\|T\|_A^2+\|S\|_A^2+\sqrt{\left(\|T\|_A^2-\|S\|_A^2\right)^2+4\|TS^{\sharp_A}\|_A^2} \right).
\end{align*}
Therefore, we prove the first inequality in Theorem \ref{moradi} by combining the last inequality together with \eqref{21.12}. Now, one observes that
\begin{align*}
&\frac{1}{2}\left(\|T\|_A^2+\|S\|_A^2+\sqrt{\left(\|T\|_A^2-\|S\|_A^2\right)^2+4\|TS^{\sharp_A}\|_A^2} \right)+2\omega_A(S^{\sharp_A}T)\\
&\leq\frac{1}{2}\left(\|T\|_A^2+\|S\|_A^2+\sqrt{\left(\|T\|_A^2-\|S\|_A^2\right)^2+4\|T\|_A^2\|S\|_A^2} \right)+2\|S^{\sharp_A}T\|_A\\
&\leq\frac{1}{2}\left(\|T\|_A^2+\|S\|_A^2+\sqrt{\left(\|T\|_A^2+\|S\|_A^2\right)^2} \right)+2\|T\|_A\|S\|_A\\
&=\|T\|_A^2+\|S\|_A^2+2\|T\|_A\|S\|_A\\
&= \left(\|T\|_A+\|S\|_A\right)^2.
\end{align*}
This immediately proves the second inequality in Theorem \ref{moradi}. Therefore, the proof is finished.
\end{proof}

\end{document}